\documentclass[12pt]{amsart}
\usepackage{amsopn,amsmath,amssymb,amsthm,eucal,url,amscd,amsgen}
\usepackage{enumerate}
\usepackage[pagebackref]{hyperref}
\usepackage[arrow]{xy} 
\usepackage{setspace} 
\usepackage{multirow} 
\usepackage{xfrac} 
\usepackage{nccmath} 
\usepackage{enumitem} 
\usepackage{xcolor} 

\makeatletter
\newcommand{\mylabel}[2]{#2\def\@currentlabel{#2}\label{#1}}
\makeatother

\newcommand{\nc}{\newcommand}
 
\nc{\bb}{\mathfrak{b}} \nc{\cc}{\mathfrak{c}} \nc{\dd}{\mathfrak{d}}
 \nc{\ggo}{\mathfrak{g}}
\nc{\hh}{\mathfrak{h}} \nc{\ii}{\mathfrak{i}}
\nc{\jj}{\mathfrak{j}} \nc{\kk}{\mathfrak{k}}
\nc{\mm}{\mathfrak{m}} \nc{\nn}{\mathfrak{n}}
\nc{\pp}{\mathfrak{p}} 
\nc{\rr}{\mathfrak{r}} \nc{\sg}{\mathfrak{s}}
\nc{\sso}{\mathfrak{so}} \nc{\spg}{\mathfrak{sp}}
\nc{\ssu}{\mathfrak{su}} \nc{\ssl}{\mathfrak{sl}}
\nc{\tog}{\mathfrak{t}} \nc{\uu}{\mathfrak{u}}
\nc{\vv}{\mathfrak{v}} \nc{\ww}{\mathfrak{w}}
\nc{\zz}{\mathfrak{z}}

\nc{\CC}{{\mathbb C}}
\nc{\DD}{{\mathbb D}}
\nc{\FF}{{\mathbb F}}
\nc{\GG}{{\mathbb G}}
\nc{\HH}{{\mathbb H}}
\nc{\II}{{\mathbb I}}
\nc{\JJ}{{\mathbb J}}
\nc{\KK}{{\mathbb K}}
\nc{\NN}{{\mathbb N}}

\nc{\RR}{{\mathbb R}}
\nc{\ZZ}{{\mathbb Z}}

\nc{\ggob}{\overline{\mathfrak{g}}}
\nc{\glg}{\mathfrak{gl}}

\nc{\pca}{\mathcal{P}} \nc{\nca}{\mathcal{N}}

\nc{\vp}{\varphi} \nc{\ddt}{\frac{{\rm d}}{{\rm d}t}}
\nc{\la}{\langle} \nc{\ra}{\rangle}
\nc{\brg}{[\,,\,]_{\ggo}}
\nc{\brv}{[\,,\,]_{\vv}}

\nc{\SO}{{\sf SO}} \nc{\Spe}{{\sf Sp}} \nc{\Sl}{{\sf Sl}}
\nc{\SU}{{\sf SU}} \nc{\Or}{{\sf O}} \nc{\U}{{\sf U}}
\nc{\Gl}{{\sf Gl}} \nc{\Se}{{\sf S}} \nc{\Cl}{{\sf Cl}}
\nc{\Spin}{{\sf Spin}} \nc{\Pin}{{\sf Pin}}

\nc{\sldr}{\operatorname{SL(2,\R)}}
\nc{\sldrt}{\operatorname{\widetilde{SL}(2,\R)}}
\nc{\Gamt}{\operatorname{\widetilde{\Gamma}}}
\nc{\alpt}{\operatorname{\widetilde{\alpha}}}
\nc{\gsldr}{\operatorname{\mathfrak{sl}(2,\R)}}
\nc{\gldr}{\operatorname{GL(2,\R)}}
\nc{\sldz}{\operatorname{SL(2,\Z)}}
\nc{\B}{\operatorname{B}}
\nc{\oscn}{\operatorname{Osc_n(\lambda_1,...,\lambda_n)}}

\nc{\ad}{\operatorname{ad}} \nc{\Ad}{\operatorname{Ad}}
\nc{\coad}{\operatorname{coad}}
\nc{\rank}{\operatorname{rank}} \nc{\Irr}{\operatorname{Irr}}
\nc{\End}{\operatorname{End}} \nc{\Aut}{\operatorname{Aut}}
\nc{\Inn}{\operatorname{Inn}} \nc{\Der}{\operatorname{Der}}
\nc{\Ker}{\operatorname{Ker}} \nc{\Iso}{\operatorname{Iso}}
\nc{\Le}{\operatorname{L}} \nc{\Fe}{\operatorname{F}}
\nc{\tr}{\operatorname{tr}}
\nc{\dif}{\operatorname{d}} \nc{\sen}{\operatorname{sen}}
\nc{\modu}{\operatorname{mod}} \nc{\Ric}{\operatorname{R}}
\nc{\Sym}{\operatorname{Sym}} \nc{\sca}{\operatorname{sc}}
\nc{\scalar}{{\sf s}} \nc{\grad}{\operatorname{grad}}
\nc{\ricci}{\operatorname{r}} \nc{\riccin}{\operatorname{Ric}}
\nc{\Lie}{\operatorname{L}} \nc{\ct}{\operatorname{T}}

\newcommand{\lela}{\left \langle}
\newcommand{\rira}{\right \rangle}
\newcommand{\bil}{\lela\,,\,\rira}

\nc{\mr}{{\mathfrak r}}
\nc{\ms}{{\mathfrak s}}
\nc{\mv}{{\mathfrak v}}
\nc{\lra}{\longrightarrow}
\nc{\R}{{\mathbb R}}
\nc{\Q}{{\mathbb Q}}
\nc{\Z}{{\mathbb Z}}

\newcommand{\mgg}{\mathfrak g}

\nc{\hs}{{G/\Gamma}}

\theoremstyle{plain}
\newtheorem{thm}{Theorem}[section]
\newtheorem{prop}[thm]{Proposition}

\newtheorem{lem}[thm]{Lemma}

\theoremstyle{definition}
\newtheorem{defn}[thm]{Definition}

\theoremstyle{remark}
\newtheorem{rem}{Remark}

\newtheorem{exa}[thm]{Example}

\newtheorem{obs}[thm]{Observations}

\begin{document}

\title[Geodesics and isometries on compact Lorentzian solvmanifolds]{Geodesics and isometries on compact Lorentzian solvmanifolds}

\begin{abstract}
The aim of this work is the study of geodesics on Lorentzian homogeneous spaces of the form $M=G/\Lambda$, where $G$ is a solvable Lie group endowed with a bi-invariant Lorentzian metric and $\Lambda < G$ is a cocompact lattice. Conditions to assert closeness of light, time or spacelike geodesics on the compact quotient spaces are given. This study implicitly needs extra information of the lattices in every case. We found conditions to assert that every lightlight geodesic on the quotient space is closed. And more important, this situation depends only on the lattice. Moreover, even in dimension four, there are examples of compact solvmanifolds for which not every lightlike geodesic is closed.
\end{abstract}

\author{Pablo Montenegro and Gabriela P. Ovando}

\let\today\relax 

\thanks{{\it (2000) Mathematics Subject Classification}: 53C50
	53C22
	22F30
	57S25}

\thanks{{\it Key words and phrases}: Lorentzian geometry, geodesics, compact solvmanifolds.}

\thanks{Partially supported by ANPCyT, SCyT-UNR.}

\address{Departamento de Matemática, ECEN - FCEIA, Universidad Nacional de Rosario, Pellegrini 250, 2000 Rosario, Santa Fe, Argentina.}

\email{gabriela@fceia.unr.edu.ar}

\maketitle

	\section{Introduction}
A Lorentzian manifold is a  connected, smooth, finite-dimensional manifold \\ $(M, \la \, , \ \ra)$, together with a Lorentzian metric, i.e. a second-order smooth
tensor field on $M$ which induces, for every $p\in M$, a bilinear form of index $1$ on the tangent
space $T_pM$ (cf. e.g. \cite{ON}). The geodesics on $M$ are the smooth curves $\gamma(t)$, satisfying the
differential equation
$$\nabla_{\gamma'(t)}\gamma'(t)=0,$$
where $\nabla$ denotes the Levi-Civita connection for $\gamma$. 	
	The genuine interest in the study of Lorentzian manifolds relies on the fact that the models of 	space-time in general relativity are four-dimensional Lorentzian manifolds. In this setting the role of a given geodesic depends on the initial. Thus, 	
a timelike geodesic, that is, for which  $\la \gamma'(t), \gamma'(t)\ra < 0$   represents the world line of a particle
	under the action of a gravitational field. While if $\la \gamma'(t), \gamma'(t)\ra = 0$, the geodesic is called lightlike or
	null, and it represents the world line of a light ray.
	
	The study of closed geodesics is a classical topic in Riemannian geometry, also treated in the Lorentzian geometry with very different techniques. Galloway in \cite{Ga}
	statet that any closed, i.e. compact without boundary, Lorentzian surface contains
	at least one closed timelike or lightlike periodic geodesic. In \cite{Su} Suhr proved that every closed Lorentzian surface contains two closed geodesics, one of
	which is definite, i.e. time- or spacelike. Indeed there are many open questions. 
	
	Lie groups are good sources to study the behavior in certain situations. In particular in this framework, in \cite{BOV} the authors show families of compact Lorentzian manifolds for which every lightlike geodesic is closed. Motivated from this result, in the present paper we investigate the situation in higher dimensions and in every lattice. Precisely take the oscillator groups of dimension $2n+2$ equipped with a Lorentzian bi-invariant metric and consider discrete subgroups such that the corresponding quotient space $M$ is compact.  We improve the result above by showing that the statement depends on the lattice.  Precisely, Theorem \ref{teoremaoscilador} shows a condition in the lattice which implies that either every lightlike geodesic in the compact manifold $M$ is closed or there is exactly one direction for which lightlike geodesics is closed and for any other direction is non closed. 
	
	To complete the study we determine the existence of closed and open timelike and spacelike geodesics. In Theorem \ref{othergeodesics} we proved that there always exist every kind of such geodesics and show explicit examples of every situation. 
	
	Finally we study isometries in the compact quotients. It was proved in \cite{BG} that the identity component of the isometry group of a pseudo-Riemannian compact space coincides with $G$, whenever $G$ is a solvable Lie group acting by isometries.  We based the study in the results obtained in \cite{Bou} where the isometry groups of the oscillators Lie groups were computed, when considered with a bi-invariant metric. By generalizing results in \cite{BOV} we proceed to compute some of the isometry groups in  compact spaces.  We notice that isometries fixing the identity element in the oscillator groups strictly include the conjugation maps (see Theorem 4.6). However to induce isometries to the quotient spaces one gets  a conjugation by an element of the normalizer of the corresponding lattice (see Proposition 4.11). On the other hand any left-translation will be induced to the quotient. Computations of the normalizer of the lattices are much more complicated in higher dimensions. In the final section we show some examples of those computations.

	\section{Lie groups with Lorentzian bi-invariant metrics}\label{section1}
	In this section one can find an introduction to general  results about Lie groups with bi-invariant Lorentzian metrics. 
	
	Let $G$ denote a (real) Lie group with Lie algebra $\mgg$. 
	A \textit{bi-invariant} metric on  $G$ is a pseudo-Riemannian metric $\bil$ for which every translation on the left $L_g$ and on the right $R_g$ by elements of the group $g\in G$, are isometries. This gives that the conjugation maps $I_g: G \to G$, $I_g(x)=gxg^{-1}$ are isometries. And so,  the differential of the Adjoint map is a linear isometry on $\mgg$, $d(I_g)_e= Ad(g)$. One has the following equivalences (see Chapter 11 in \cite{ON}):
	\begin{enumerate}\label{[(i)]}
		\item $\bil$ bi-invariant;
		\item $\bil$ Ad($G$)-invariant;
		\item $\lela [X, Y], Z\rira + \lela Y, [X, Z]\rira= 0$ for all $X, Y, Z \in\mgg$;
		\item the geodesics of $G$ starting at the identity element $e$ are the one-parameter subgroups of $G$, that is:
		\begin{equation}\label{onepara}
			\alpha(t)=\exp(tX), \qquad \mbox{ for }X  \in \mgg, 
		\end{equation}
		and the geodesic through $g\in G$ with initial left-invariant vector $X$ is given by the translation of the curve above, that is $g\exp(tX)$. 
	\end{enumerate}
	If the bi-invariant metric on a Lie group $G$ of dimension $n$ has signature $(1,n-1)$, the metric is called a {\em Lorentzian metric}. 
	
A given vector field  $X\in TG$  is called 
	\begin{itemize}
		\item {\em spacelike} whenever $\lela X,X \rira >0$;
		\item {\em timelike} whenever $\lela X,X \rira < 0$;
		\item {\em lightlike} or {\em null}  if  $\lela X,X \rira = 0$.
	\end{itemize}
More generally this is extended to geodesics: 	a geodesic on $G$ with  initial condition $X$, namely $\gamma_X(t)$, is called {\em spacelike, timelike or lightlike} if $X$ is in the respective class above.

	Examples of Lie groups with bi-invariant Lorentzian metrics arise from the so called  \textit{oscillator groups}. Denoted by $\oscn$, an oscillator Lie group is the simply connected Lie group with real Lie algebra of dimension $2n+2$, namely $\mathfrak{osc}_n(\lambda_1,...,\lambda_n)$, with $\lambda_i\in \R_{>0}$. This Lie algebra is  spanned by the basis $Z$,$\{X_i,Y_i\}_{i=1}^n$, $T$ satisfying the non-trivial Lie bracket relations
	\[ [X_i,Y_i]=Z, \quad [T, X_i]=\lambda_i Y_i, \quad  [T, Y_i]=- \lambda_i X_i.  \]
Denote by $\lela\,,\, \rira$  the ad-invariant  metric on $\mathfrak{osc}_n(\lambda_1,...,\lambda_n)$ with  the non-zero relations
	\[ \lambda_i \lela X_i,X_i \rira  = \lambda_i \lela Y_i,Y_i \rira = \lela Z,T  \rira  = 1.\]
	
	The oscillator Lie groups   have the differential structure of $\R \times \R^{2n} \times \R$ with the following group product
	\begin{equation*}
		(z_1,v_1,t_1) . (z_2,v_2,t_2)=(z_1+z_2+\frac{1}{2}v_1^{\tau}J R(t_1)v_2,v_1+R(t_1)v_2,t_1+t_2),
	\end{equation*}
	\[\text{where \,  } 
	R(t_1)= e^{t_1 N_{\lambda}}, \, N_\lambda=\left( \begin{matrix}
		J_{\lambda_1} &  & \mathbf{0} \\
		& \ddots & \\
		\mathbf{0} & & J_{\lambda_n}
	\end{matrix} \right),\, 
	J_{\lambda_i}=\left( \begin{matrix}
		0 & -\lambda_i \\
		\lambda_i & 0
	\end{matrix} \right), \, J = N_{(-1, \hdots, -1)}.
	\]
	
	for $v_1, v_2\in \RR^{2n}$. By $v^{\tau}$ we denote the transpose of $v$. 
	Take the corresponding left-invariant metric on the Lie group, which for usual  coordinates for $i=1, \hdots, n$: $z, x_i,y_i, t$ in $\R^{2n+2}$  can be written as
	\begin{equation}\label{metricosc}
		g=dt (dz +\sum_{j=1}^{n} y_j dx_j+ x_j dy_j) +\sum_{j=1}^{n}\frac1{\lambda_j}(dx_j^2+dy_j^2).
	\end{equation}
	
	The Christoffel symbols corresponding to the metric above follow
	
	\[ \Gamma^1_{2n+2 \,\, 2i}=-\frac{x_{i} \lambda_{i}}{4} \quad \Gamma^1_{2n+2 \,\, 2i+1}=-\frac{y_{i} \lambda_{i}}{4}, \quad i=1,..., n \]
	
	\[ \Gamma^{2i}_{2n+2 \,\, 2i}=\frac{\lambda_i}{2} \quad \Gamma^{2i+1}_{2n+2 \,\, 2i}=-\frac{\lambda_i}{2}, \quad i=1,..., n \]
	
	being the others trivial and following  symmetry relations. 
	
	The resulting equations for the geodesics can be written in the usual  coordinates of $\RR^{2n}$ as:
	\begin{equation}\label{geodcomp}
		\begin{array}{rcl}
			z''(s)&= & \frac{ t'(s)}{2}\sum_{k=1}^{n}  \lambda_k \left( x_k'(s) x_k(s)+y_k'(s) y_k(s) \right) \\ \vspace{.2cm}
			x_i''(s)&=&-\lambda_i y_i'(s) t'(s),\\ \vspace{.2cm}
			y_i''(s)&=&\lambda_i x_i'(s) t'(s),\\ \vspace{.2cm}
			t''(s)&=&0,
		\end{array}
	\end{equation}
		which follows from the general geodesic equation, $\frac{d^2 \gamma^k}{d t^2} + \sum_{i,j} \Gamma^k_{i j}(\gamma) \frac{d \gamma^i}{dt} \frac{d \gamma^j}{dt} = 0$ (see \cite{ON}).

	In particular, those geodesics $\gamma_X(s)=(z(s), (x_j(s),y_j(s)),t(s))$, $j=1, \hdots n$ starting at the identity element  with initial condition $X =  d \ Z + \sum_j (b_j X_j + c_j Y_j) + a T$ are:
	
	\begin{itemize}
		\item for $a \neq 0$:
			\begin{eqnarray} \label{geo_osc_1}
			z(s)&= & \left(d + \frac{1}{2 a} \sum_{k=1}^{n} \frac{ b_{k}^{2}+c_k^{2}}{\lambda_k}\right)s- \frac{1}{2 a^{2}} \left(  \sum_{k=1}^{n} \frac{b_{j}^{2}+c_j^2}{\lambda_k^{2}} \sin(\lambda_k a s) \right),\\
			x_j(s)&=&  \frac{1}{a \lambda_j} \left(   {b_j}sin(\lambda_j a s)+{c_j}cos(\lambda_j a s)-{c_j} \right),\\
			y_j(s)&=&  \frac{1}{a \lambda_j}  \left(    -{b_j}cos(\lambda_j a s)+{c_j} sin(\lambda_j a s)+{b_j} \right),\\ 
			t(s)&=&a s,
		\end{eqnarray}
		\item while for $a=0$, one has:
		\begin{equation}\label{geo2}
			(z,(x_j,y_j),t)(s)=(ds,(b_j s,c_j s),0). 
		\end{equation}
		
	\end{itemize}
	
	It is not hard to check that for the initial velocity $X \in \mathfrak{osc}_n(\lambda_1,...,\lambda_n)$ as above, the corresponding geodesic is:
	\begin{itemize}
		\item lightlike if $2 a d + \sum_{k=1}^{n} \frac{b_k^2+c_k^2}{\lambda_k} = 0$,
		\item timelike if $2 a d + \sum_{k=1}^{n} \frac{b_k^2+c_k^2}{\lambda_k} < 0$, 
		\item or spacelike if $2 a d + \sum_{k=1}^{n} \frac{b_k^2+c_k^2}{\lambda_k} > 0$.
	\end{itemize}
	
	Note that the oscillator  Lie groups are also complete spaces. 
	
	\smallskip
	
	\begin{rem} Medina and Revoy  in \cite{Me,MeRe} proved that the Lie algebras $\mathfrak{osc}_n(\lambda_1,...,\lambda_n)$ ($\lambda_i > 0$) and $\gsldr$ are the only indecomposable ones admitting a Lorentzian ad-invariant metric. Recall that a Lie algebra provided with a metric is called \textbf{indecomposable} if the restriction of the metric to any proper ideal is degenerate. 
	\end{rem}

	\subsection{Quotient spaces} 

Let $G$ denote a Lie group and let  $\Gamma\subset G$ be a discrete cocompact subgroup. 	The quotient space $M=G/\Gamma$ consists of elements of the form $g\Gamma$ with $g \in G$. Since $\Gamma$ is closed, there exists a unique manifold structure on $M$ for which the canonical projection $g \mapsto g\Gamma$ is a smooth submersion (see \cite{Hel}). Finally, the geometry of $M$ is provided by requiring the projection, named $\pi$, to be a local isometry. Whenever the Lie group  $G$ is provided with a Lorentzian metric, $(G,\pi)$ is called  a \textit{Lorentzian covering} of $M$.

Assume $G$ is equipped with a bi-invariant metric. It follows  that the geodesics of $M$ starting at $o:=\pi(e)$ are of the form $\hat{\alpha}=\pi(\alpha(t))$, where $\alpha$ is a one parameter subgroup of $G$ (see \cite{ON}). In addition to this, $G$ acts on $M$ by the "translations on the left" which are isometries:
	\begin{eqnarray*}
		\tau_g : M \rightarrow M\qquad \mbox{given by} \quad 
		\tau_g(h\Gamma):=gh\Gamma,
	\end{eqnarray*}
	 showing that $M=G/\Gamma$ is a homogeneous space. 
	
One can notice that: 
	
	\begin{enumerate}
		\item A geodesic of $G/\Gamma$ starting at $g\Gamma$ is the translation via $\tau_g$ of some geodesic starting at $o$. \label{punto1}
		\item Every geodesic in $G/\Gamma$ is the projection via $\pi$ of some geodesic in $G$.\label{punto2}
		\item Lighlike, timelike and spacelike geodesics of $G$ project to lightlike, timelike and spacelike geodesics of $M$ respectively.
	\end{enumerate}
	
Since $\pi \circ L_g = \tau_g \pi$, one gets that $\tau(g)\pi\circ \alpha=\pi\circ L_g \circ \alpha$ for a curve $\alpha:(a,b)\to G$ starting at the identity element  $e\in G$. 

	 A curve $\beta:(a,b)\to G$ (or to $M$) is said {\em non-simple } when it passes through a same point more than once, that is, there exist $t_2\neq  t_1$ such that $\beta(t_1)=\beta(t_2)$. 
	\begin{enumerate}
		\setcounter{enumi}{3}
		
		\item A geodesic $\alpha: (-\varepsilon, \varepsilon) \to G$, with $\varepsilon>0$ and   $\alpha(0)=e$   giving rise the the curve $\pi\circ \alpha$ in $M$  is non-simple in $M$ if and only if  $\alpha(t) \in \Gamma$ for some $t>0$.\label{punto4}
	\end{enumerate}
	In particular the projection of a non-simple geodesic in $G$ is always a closed curve in $M$.

	A final result for non-simple geodesics comes from the following lemma which, when combined with item (\ref{punto4}) states that every non-simple geodesic in the quotient manifold is actually a periodic curve. 
	
	\begin{lem}\cite{BOV}  Let $G$ be a Lie group, let $K < G$ be any closed Lie  subgroup of $G$ such that  $\pi: G \to G/K$ denotes the 
		usual projection. Let $\alpha: \RR \to G$ denote a  one-parameter subgroup of $G$.
		If $\pi \circ \alpha$ is non-simple in $G/K$ then it is periodic.
	\end{lem}
	
In this paper, {\em closed} geodesics will be periodic ones.

	\section{The solvmanifolds from the Oscillator groups}\label{sectionosc}
	
	This section is concerned with the study of geodesics of Lorentian compact spaces  $$M=\oscn/ \Gamma,$$ where $\Gamma$ is a cocompact lattice in $\oscn$. The following results shows a condition to construct such lattices. 
	
	\begin{lem}\cite{MeRe}\label{lema_medina}
		An oscillator group $\oscn$ admits a lattice if and only if the numbers $\lambda_j$ generate an additive discrete subgroup of $\R$.
	\end{lem}
	
	In the demonstration of the previous lemma it is  shown that for a lattice $\Gamma$, the set $\mathrm{T}(\Gamma):=\{ t \in \R : (z,u,t) \in \Gamma \}$ is an additive discrete subgroup of $\RR$. 
	 Let $t_0$ denote the positive generator of $\mathrm{T}(\Gamma)$. 
	
	Notice that for $(w,b,0) \in \Gamma$, the set of elements in the lattice
	\begin{equation*}
		\{ (z,u,t_0)^n.(w,b,0).(z,u,t_0)^{-n}=(w,e^{n t_0 N_{\lambda}}b,0) : n \in \mathbb{N} \}
	\end{equation*}
	 is a finite set, since they are elements of a discrete cocompact lattice. Furthermore there is a smaller positive integer $K_0$ such that $e^{K_0 t_0 N_{\lambda}} = Id$. In particular it follows that $t_0$ satisfies
	\begin{equation} \label{oscilator-N}
		t_0=\frac{2 \pi k_i}{\mathrm{K_0} \lambda_i},
	\end{equation}
	for some integers $k_i$ with $i=1, ..., n$.\\
	
	In \cite{MF}, Fischer introduced a family of Lie groups named $Osc_n(\omega_r, B_r)$ defined by an element $r=(r_1, ..., r_n) \in \mathbb{N}^n$ such that $r_i | r_{i+1}$. Denote by $\omega_r(u,v):=u^TN_{-r}v$ the symplectic form on $\R^{2n}$ and by $B_r \in GL(2n, \R)$ the linear transformation satisfying
	
	\begin{itemize}
		\item $\omega_r(B_r.,.)$ is symmetric and negative definite
		\item $e^{B_r} \in SL(2n,\Z)$.
	\end{itemize}
	
	The group operation for $Osc_n(\omega_r, B_r)$ with base on the manifold $\R \oplus \R^{2n} \oplus \R$ is given by
	
	\begin{equation}
		(z_1,v_1,t_1) . (z_2,v_2,t_2)=(z_1+z_2+\frac{1}{2}v_1^{T}N_{-r} e^{t_1 B_r}v_2,v_1+e^{t_1 B_r}v_2,t_1+t_2).
	\end{equation}
	
	Let $L(\xi_0)$ be the subgroups generated by $$\{ (1,0,0),(0,e_i,0),(0,\xi_0, 1) \}$$ where $\xi_0$ is an element in $\R^{2n}$ such that the above subgroup is a lattice. 
	
In particular, according to Example 3.1 of \cite{MF}, the element $\xi_0$ verifies the following condition 
	\begin{equation}\label{xi-condition}
		(\omega_r(\xi_0, e^{B_r}e_i), e^{B_r} e_i, 0) \in \,\, <\{ (1,0,0),(0,e_i,0) \}>
	\end{equation}

We assert that  the lattices $L(\xi_0)$ of $Osc_n(w_r, B_r)$ can be associated to lattices of $\oscn$. In fact, for every lattice $\Gamma$ of $\oscn$ there exists a group $Osc_n(\omega_r, B_r)$, $\xi_0 \in \R^{2n}$ and an isomorphism $\Phi: \oscn \rightarrow Osc_n(\omega_r, B_r)$ such that $\Phi(\Gamma) = L(\xi_0)$ (see Theorem 5 of \cite{MF}). 	
	The explicit definition of $\Phi$ can be found in the proof of the mentioned theorem. Moreover, the following property of this isomorphism holds:
	
	\begin{equation} \label{condition-exp}
		\Phi^{-1}(z,0,t) = (w z, 0, \widetilde{t_0} t ) \mbox{ whenever } e^{t B_r} = Id,
		\end{equation}
    where $\widetilde{t_0}$ is either $\frac{1}{t_0}$ or $-\frac{1}{t_0}$.

	Additionally it is shown that $B_r := \pm t_0 S N_{\lambda} S^{-1}$ \footnote{this follows by  noticing that $\oscn = Osc_n(w_{1}, N_{\lambda})$.}, for some invertible matrix $S$. 
	
	\begin{lem}\cite{MF}\label{oscilador-elementos}  
		Let $\Gamma$ be any lattice of $\oscn$. Then:
		\begin{enumerate}
			\item There always exists \, $w \neq 0 \in \R$ such that $(w,0,0) \in \Gamma$.
			\item If $K_0=1$, implicitly defined in equation \ref{oscilator-N}, then there exists an element in $\Gamma$ of the form $\gamma = (z, 0, t)$, where $z$ and $t$ are non-zero.
		\end{enumerate}
	\end{lem} 
	
	\begin{proof}
		
		Since $\Gamma = \Phi^{-1}(L(\xi_0))$, then $\Phi^{-1}(1,0,0) = (w,0,0) \in \Gamma$, according to (\ref{condition-exp}), with $w \neq 0 $; this proves the first part of the lemma.
		
		The second part is proved noticing first that the condition $t_0 = \frac{2 \pi k_i}{\lambda_i}$ corresponds to lattices such that $e^{tN_\lambda} = Id$ for any $t \in \mathrm{T}(\Gamma)$, therefore $$e^{B_r} = Se^{t_0N_\lambda}S^{-1}=Id.$$
		
		The latter equation, together with the fact that $r_i | r_{i+1}$ applied in Condition (\ref{xi-condition}) gives the following property: $$\xi_0=(\frac{z_1}{r_1},\frac{z_2}{r_1}, \frac{z_3}{r_1 k_2}, \frac{z_4}{r_1 k_2}, ..., \frac{z_{2n-1}}{r_1 k_2 k_3 ... k_n}, \frac{z_{2n}}{r_1 k_2 k_3 ... k_n} ), \quad \mbox{for some } z_i \in \mathbb{N},$$ and it can be verified that $$(0,\xi_0,1)^{r_1 k_2 k_3 ... k_n} \in \Q \times \Z^{2n+1}. $$
		
		Since the $2n$-components in $\mathbb R^{2n}$ of the result are integers, every element of this form  can be multiplied conveniently by $(0, \pm e_i, 0) \in L(\xi_0)$ to obtain $(q_1, 0, t_1) \in \Gamma$, for some $q_x \in \Q$. Then for some integer $y_1$, $(q_1, 0, t_1)^{y_1} = (y_1 q_1, 0, y_1 t_1) \in \Z^{2n+1}$. 
		
		Finally, since $(\pm 1,0,0) \, L(\xi_0)$, after convenient multiplications one can construct an element $(1,0,k)$ in $L(\xi_0)$ such that, $\Phi^{-1}(1,0,k) = (w,0,\widetilde{t_0} k)$. 
		\end{proof}

	\begin{obs}\label{obs-osc}
		Let $\Gamma$ be a lattice of the oscillator group $\oscn$ with $t_0=\frac{2\pi k_i}{\mathrm{K_0} \lambda_i}$ as in Equation \eqref{oscilator-N}, notice that:
		\begin{itemize}
			\item The lightlike geodesics in $\oscn$ with $a=0$ \eqref{geo2}, verify $b_j=c_j=0$ for all $j=1,...,n$. Consequently, they take the form $ \alpha_d(s)=(ds,0,0)$ and intersect $\Gamma$ because there exists $w > 0$ $(w,0,0) \in \Gamma$ according to last Lemma. This means that  $\alpha(\tilde{s})=(w,0,0)$ for some $\tilde{s} > 0$.
			\item The lightlike geodesics with $a \neq 0$ verify $\alpha(\frac{\mathrm{K_0} t_0}{a}) = (0,0,\mathrm{K_0} t_0)$, see the expressions \eqref{geo_osc_1}. If the lattice  $\Gamma$ contains an element of the form $(0,0,\hat{t})$ with $\hat{t}=p t_0$ for some $p \in \mathbb{Z}$, then $$\alpha(p \mathrm{K_0} t_0) = (\alpha(\mathrm{K_0} t_0))^p = (0,0,\hat{t})^p \in \Gamma.$$ 
		\end{itemize}
	\end{obs}	
		
		\begin{thm}\label{teoremaoscilador}
			Let $\Gamma$ be a cocompact lattice of $\oscn$, and consider the compact Lorentzian manifold $M=\oscn/\Gamma$. Then only one of the following situations occurs
			\begin{itemize}
				\item either $\Gamma$ contains an element of the form $(0,0,t_0),$ for some $t_0\in \RR, t_0 \neq 0$ and in this situation every lightlike geodesic of $M$ is closed;
				\item or, for any $t \neq 0$, one has $(0,0,t) \notin \Gamma$. In this case,   at every point in $M$ there is exactly one direction for which all lightlike geodesics of $M$ are closed and for any other direction they are non-closed. This direction is spanned by the lightlike element $Z \in \mathfrak{osc}_n(\lambda_1, ..., \lambda_n)$.		
			\end{itemize}
			
		\end{thm}
		
		\begin{proof}
			Recall that it suffices to study the geodesics starting at $o:=\pi(e)$ and that every geodesic $\hat{\alpha}$ is the projection of some geodesic, $\alpha$, on $\oscn$: $\hat{\alpha}=\pi(\alpha)$ with $\alpha(0)=e$. Also, $\hat{\alpha}$ is closed in $M$ if $\alpha(s) \in \Gamma$ for some $s>0$.
			
			As observed in \ref{obs-osc}, all lightlike geodesics of the form $\pi((ds,0,0))$ are closed on the compact space $M$, and so geodesics with  this direction will  always be closed. Therefore, to prove the theorem it may hold that any lightlike geodesic with different starting direction is either closed or simple. \footnote{{\color{red} esta bien?...}}
			
			Let $\alpha$ be a lightlike geodesic with a direction $X$, wich is linear independent with $Z$, and suppose that  it is closed. This means that there exists some $\gamma=(z,u,t) \in \Gamma$ for which $\alpha(s)=\gamma$ for some $s>0$. Since the curve $\alpha$ is a one-parameter subgroup of $\oscn$, then for any integer $m$: $\alpha(m s)=\gamma^m$, which is an element of $\Gamma$. Recall also that since $t \in \mathrm{T}(\Gamma)$ it is of the form $r t_0$ for some integer $r$ and $t_0=\frac{2 \pi k_i}{K_0 \lambda_i}$. Finally, since $s=\frac{t}{a}$, one can compute that $\gamma^{K_0} = \alpha(K_0 \frac{t}{a}) = (0,0,K_0 t) = (0,0,K_0 r t_0)$, and therefore, since $K_0 r$ is an integer, this element is in the lattice and every lightlike geodesic of $M$ is closed.
			
			In conclusion, when an element of the form $(0,0,k t_0)$ is in the lattice every lightlike geodesic of $M$ is closed, otherwise $\hat{\alpha}_d(s)=\pi(ds,0,0)$ are the only closed  geodesics at $\pi(e)$.			
		\end{proof}

		\begin{exa}\label{Lattice4} Both situations stated in the above theorem are possible. Take for instance the three families of cocompact lattices constructed in \cite{BOV} for  $Osc_1(1)$,
			\begin{eqnarray*} \label{geodlight}
				\Lambda_{n,0}&=&\frac{1}{2n}\Z \times \Z \times \Z \times 2 \pi \Z,\\
				\Lambda_{n,\pi}&=&\frac{1}{2n}\Z \times \Z \times \Z \times \pi \Z,\\
				\Lambda_{n,\frac{\pi}{2}}&=&\frac{1}{2n}\Z \times \Z \times \Z \times \frac{\pi}{2} \Z,
			\end{eqnarray*}
			where $n \in \mathbb{N}$, for which the authors proved that all lightlike geodesics of $M_{n,0}=Osc_1(1)/\Lambda_{n,0}, M_{n,\pi}=Osc_1(1)/\Lambda_{n,\pi}$ and $M_{n,\pi/2}=Osc_1(1)/\Lambda_{n,\pi/2}$ are closed. However other  lattices can be obtained by noticing that
			\begin{eqnarray*}
				\phi_m &:& Osc_1(1) \rightarrow Osc_1(1)\\
				\phi_m(z,x,y,t)&=&(z+mt,x,y,t) \textrm{,    $m \in \R$}
			\end{eqnarray*}
			are  automorphisms of $Osc_1(1)$. So, the  lattices $\phi(\Lambda_{n,\bullet})$ most likely do not contain an element of the form $(0,0, t)$. For example, given an integer $p \neq 0$, the lattice $\phi_p(\Lambda_n,0)$ does not contain such element since $\frac{a}{2 n}+ p \, 2 \pi b = 0$ has no solution for integers $a,b$. Thus, for these lattices not every lightlike geodesic is closed. \\
			
		\end{exa}
  
        \begin{rem}
            Each lattice, $L(\xi_0)$, in Table 6 of \cite{MF} corresponds to a lattice $\Gamma$ of $Osc_1(1)$ where all the lightlike geodesics on $Osc_1(1)/\Gamma$ are closed. Such correspondence is given by the inverse of the following group isomorphism \cite{MF}
              \begin{eqnarray}
                  \phi: Osc_1(1) \rightarrow Osc_1(w_r, B_r) \\
                  (z,v,t) \mapsto (r z, T_{x,y} v, t/\lambda),
              \end{eqnarray}
              with $$T_{x,y} = \left( \begin{matrix}
				-\sqrt{y} & -\frac{x}{\sqrt{y}} \\
				0 & -\frac{1}{\sqrt{y}} \\
			\end{matrix} \right).$$

   To see  the lightlike geodesics on $Osc_1(1)/\Gamma$ for lattices $\Gamma := \phi^{-1}(L(\xi_0))$, one can first notice the following property $$(0,0,\lambda k) = \phi^{-1}(0,0,k), \forall k \in \Z. $$ Therefore if $(0,0,k)$ in $L(\xi_0)$ the observation will be proved, according to theorem (3.4).

   Prove that $(0,0,k) \in L(\xi_0)$ for some $k \in \ZZ - \{ 0 \}$. Notice first that since $\xi_0 \in \Q^2$ \cite{MF}, and so $(0,\xi_0,1)^{y_1}$ for any $y_1 \in \ZZ$. Additionaly, since there exist $N$ such that $e^{N B_r} = Id$ (see derivation of Equation \eqref{oscilator-N}), one has  $(0,\xi_0,1)^{N y_1 n_1} = (x_2,v_2,t_2) \in \ZZ^4$ for some $n_1 \in \NN$. Finally $(x_2, v_2, t_2) . (0,1,0)^{-v_2} . (1, 0, 0)^{-x_1} = (0,0,t_2)$.
        \end{rem}
		
		To  study  timelike and spacelike geodesics on the compact spaces, one needs to  consider the geodesics on $\oscn$ starting at the identity element as in (\ref{geo_osc_1}).
		
		 Let  $(d,b_j,c_j,a)\in \mathfrak{osc}_n$ be the initial velocity of a geodesic where $a\neq0$ and let $\hat{\gamma}=(\hat{z}, \hat{\eta}, \hat{t})$ be an element of the lattice $\Gamma$. Assume that $\alpha(\hat{t}/a)=\gamma$ with $\hat{t}/a > 0$. In this situation, it holds		
		\begin{equation}\label{oscilador_geos_1}
			\left( \begin{matrix}
				\sin{\lambda_j \hat{t}} & \cos{\lambda_j \hat{t}} -1 \\
				1 - \cos{\lambda_j \hat{t}} & \sin{\lambda_j \hat{t}} \\
			\end{matrix} \right)
			\left( \begin{matrix}
				b_j \\
				c_j \\
			\end{matrix} \right)=
			\left( \begin{matrix}
				\hat{b_j} \\
				\hat{c_j}
			\end{matrix} \right).
		\end{equation}
		
		\begin{equation}\label{oscilador_geos_2}
			\hat{z} =  \left(d + \frac{1}{2 a} \sum_{k=1}^{n} \frac{ b_{k}^{2}+c_k^{2}}{\lambda_k}\right)\frac{\hat{t}}{a}- \frac{1}{2 a^{2}}  \sum_{k=1}^{n} \frac{b_{j}^{2}+c_j^2}{\lambda_k^{2}} \sin(\lambda_k \hat{t}).
		\end{equation}
		
		These expressions are used to prove the first part of the following theorem.
		
			\begin{thm}\label{othergeodesics}
			For any lattice $\Gamma$ of $\oscn$ there are both closed and open timelike and spacelike geodesics on the compact space $\oscn / \Gamma$.
		\end{thm}
		
		\begin{proof}
						1) Existence of closed timelike and spacelike geodesics: As seen above,  having closed timelike or spacelike geodesics of $\oscn/\Gamma$ is equivalent to have timelike or spacelike geodesics of $\oscn$ that intersect the lattice $\Gamma$ at some positive time.
			
			Take an element $(w,0,0) \in \Gamma$ with $w>0$ (see Lemma \ref{oscilador-elementos}), and consider any element $\gamma=(z, \eta, t) \in \Gamma$. Thus,  by multiplying those elements one gets  
			
			$(w,0,0)^m.({z}, {\eta},{t})=(m w+{z},{\eta}, {t})$ for any $m \in \mathbb{Z}$. 
			
			Consider now the following two posibilities for $\mathrm{K_0}$ (Equation \eqref{oscilator-N}):
			
			\begin{itemize}
				\item Case $\mathrm{K_0} = 1$. This is the case of the second item of Lemma (\ref{oscilador-elementos}). Thus,  there exists $\gamma = (z,0,t) \in \Gamma$ with $z t \neq 0$. Let $\gamma_m := (m w+z, 0, t)=(w,0,0)^m.(z,0,t)$ and consider the geodesic $\alpha_m$ with initial velocity $X =  d \ Z + \sum_j (b_j X_j + c_j Y_j) + a T$ satisfying 
				$$a=t,\quad b_j=c_j=0, \quad d_m = m w + z.$$ 
				 It follows from equations above that $\alpha_m(1)=\gamma_m$ (in fact, for $\mathrm{K_0}=1$ the matrix in \eqref{oscilador_geos_1} is trivial). Finally $\alpha_m$ is timelike or spacelike depending on  $\frac{mw+z}{t}$ is negative or positive respectively; and  any case can be achieved by choosing $m$ conveniently. 
				
				\item Case $\mathrm{K_0}>1$. Consider $\gamma=(x,u,(\mathrm{K_0}-1)t_0)$, and for $m\in \Z$ define 
				$$\gamma_m := {(m \omega + x, u, (\mathrm{K_0}-1) t_0)} = (w,0,0)^m.(x,u,(\mathrm{K_0}-1) t_0).$$
				 For every $\gamma_m$ the matrix in Equation \eqref{oscilador_geos_1} is non-singular (because if $\lambda_j (\mathrm{K_0}-1) t_0 = 2 \pi s_j$ for integers $s_j$ one gets $t_0 = \frac{2 \pi (\mathrm{K_0}-1)}{\lambda_j}$ meaning $\mathrm{K_0}=1$, which is a contradiction). Therefore Equation \eqref{oscilador_geos_1} gives unique solutions $b_j,c_j$, independent of $m$. Then setting $a=(\mathrm{K_0}-1) t_0$ and by solving Equation \eqref{oscilador_geos_2} for $d=d_m$ one gets parameters $a,b_j,c_j,d_m$ such that $\alpha_m(1) = \gamma_m$.     
				Finally these geodesics are closed in the quotient and are timelike or spacelike according to the expression  $$ 2 p t_0 (mw+x) - \sum_{k=1}^{n} \frac{b_j^2 + c_j^2}{\lambda_k^2}\sin(\lambda_k (K_0-1) t_0) $$ is negative or positive respectively. Both cases are achievable by choosing different values of $m$.
				
			\end{itemize}

			2) Existence of open timelike and open spacelike geodesics:
			
		Consider the elements of the lattice of the form  $\hat{\gamma}=(\hat{z},\hat{u}, p \mathrm{K_0} t_0) \in \Gamma$. Those elements can be obtained by considering the $\mathrm{K_0}$th power of  any element with non-null $t$-component. Let $\hat{s}$ such that $\alpha(\hat{s})=(z(\hat{s}),u(\hat{s}),t(\hat{s})) = \hat{\gamma}$, where $\alpha$ is a geodesic of $\oscn$ (with $a \neq 0$). Then it must be $t(\hat{s}) = a \hat{s} = p \mathrm{K_0} t_0$, which implies $u(\hat{s})=0$ and $z(\hat{s}) = (d + \frac{1}{2 a} \sum^n_{k=1} \frac{b_k^2+ c_k^2}{\lambda_k}) \frac{p \mathrm{K_0}t_0}{a}$, where $a, b_k, c_k, d$ define the initial velocity of $\alpha$. \\
			
			Let $X =  d' \ Z + \sum_j (b_j X_j + c_j Y_j) + a T$ de the initial condition of the geodesic $\alpha_{d'}$. For any $\varepsilon >0$ consider the interval  $I_d := [d, d+ \epsilon]$. Take  $d'\in I_d$ and assume that the geodesic $\alpha_{d'}$ intersects the lattice  at $s'$, say $\alpha_{d'}(s') \in \Gamma$. Then it must hold $t'(s') = r' t_0$ for some integer $r'$. Now, define a function $F: I_d \to \ZZ$, that sends $d'\to r'$. 
		Clearly there exists and element denoted by $r_{\infty}\in \ZZ$ such that  the  preimage  $F^{-1}(r_{\infty})$ is an infinite set. Since $F^{-1}(r_{\infty})\subset I_d$, then this set is bounded and it must contain a convergent sequence, namely $\{d'_n\}$.
			
			Take the elements in the lattice $\Gamma$ given by
			\begin{eqnarray*}
				\alpha_{d'_n}(\frac{r_{\infty} t_0}{a}) = ( (d'_n + \frac{1}{2 a} \sum_{k=1}^{n} \frac{ b_{k}^{2}+c_k^{2}}{\lambda_k})\frac{\hat{t}}{a}- \frac{1}{2 a^{2}} (  \sum_{k=1}^{n} \frac{b_{j}^{2}+c_j^2}{\lambda_k^{2}} \sin(\lambda_k \hat{t}) ), \\ 
				R_1(r_{\infty} t_0)\left( \begin{matrix}
					b_1 \\
					c_1 \\
				\end{matrix} \right),..., R_n(r_{\infty} t_0)\left( \begin{matrix}
					b_n \\
					c_n \\
				\end{matrix} \right), \\     
				r_\infty t_0 ),
			\end{eqnarray*}
			
			with $R_j(x) := \left( \begin{matrix}
				\sin{(\lambda_j x)} & \cos{(\lambda_j x)} -1 \\
				1 - \cos{(\lambda_j x)} & \sin{(\lambda_j x)} \\
			\end{matrix} \right)
			\left( \begin{matrix}
				b_j \\
				c_j \\
			\end{matrix} \right)$, see Equation \eqref{oscilador_geos_1}. Since $\{d'_n\}_n$ is convergent, the resulting sequence $\{ \alpha_{d'_n}(\frac{r_{\infty} t_0}{a}) \}_n$ also converges. This is a contradiction since $\Gamma$ is discrete.
			
		\end{proof}

	\subsection{Remarks}
	Consider the group $G=\oscn \times \R$. This Lie group is simply connected, has a bi-invariant Lorentzian metric and it admits cocompact lattices. However its Lie algebra is not indecomposable. This last fact affects the geometry of $M=G/\Gamma$, for $\Gamma$ a cocompact lattice. Take for instance the lightlike geodesics of $M$. Take a geodesic of $\R$ of the form $\gamma(t)=\beta t$, and let $\alpha$ be a geodesic of $\oscn$. Then the curve  $c(t)=(\alpha(t),\gamma(t))$ is lightlike on $M$ if the following equality holds
	\begin{equation}\label{remarkosc}
		0 = \, <\alpha'(0),\alpha'(0)> \, + \, r^2, 
	\end{equation}
	where $\la \,,\,\ra$ is the metric of the oscilator (\ref{metricosc}) at the identity. It follows that $\alpha$ must be a timelike geodesic of $\oscn$. Choose $\Lambda_{n,0}$  a lattice of $Osc_1(1)$, and $w \Z$ a lattice of $\R$, which is the case for any real $w \neq 0$. Thus,  $\Lambda_{n,0} \times w \Z$ is a lattice of $G=Osc_1(1) \times \R$.
	
	The lightlike condition in this case, explicitly Equation (\ref{remarkosc}), gives
	
	\begin{equation*}
		0 = \, 2 a d + \frac{b+c}{2a} + \, r^2. 
	\end{equation*}
	
	Should $c(t)$ be a closed lightlike geodesic of $G$ then there exists some  $s>0$, such that $\alpha(s) \in \Lambda_{n,0}$ and $r(s) \in w \Z$. From the equations for  $\alpha$ it follows that $s=\frac{2 \pi k}{a}$ for some $k \in \Z$ and $z(\frac{2 \pi k}{a})=(d+\frac{b+c}{2 a})(\frac{2 \pi k}{a})=\frac{m}{2n}$ for some $m \in \Z$. This can be reduced to 
	
	\begin{equation*}
		r^2 = \, - \frac{a^2 m}{2 \pi k},
	\end{equation*}
	
	also it must be that $\gamma(\frac{2 \pi k}{a}) = r \frac{2 \pi k}{a} = w z$ $\rightarrow$ $r^2 = \frac{a^2 z^2 w^2}{(2 \pi k)^2}$. Then,  since $a \neq 0$ one may have
	
	\begin{equation*}
		w^2 = -\frac{2 \pi k m}{z^2}.
	\end{equation*}
	
	In conclusion, since it is possible to choose $w$ for which the equality above never holds  for any $k,z \in \Z$, lightlike geodesics of $Osc_1(1) \times \R/\Gamma$ are never closed. Take for instance  $w = e\in \R$.


	\section{Isometries of the oscillator groups and compact quotients}
	In this section we study the isometries of the oscillator groups and their compact quotients. 
	
An isometry of  a Lie group $(G, \lela\,,\,\rira)$ is a differentiable diffeomorphism $\Psi:G \to G$ such that its differential preserves the metric at every point.  
    The group of isometries of a Lie group with a left-invariant pseudo-Riemannian  metric can be expressed as $Iso(G) = L(G)F(G)$, where $L(G)$ represents the subgroup of left-translations and $F(G)$ is the set of  those isometries that fix the identity element of $G$.  Since every isometry $\psi$ decomposes as $\psi = L_g \circ \phi$ where $\phi(e)=e$, the main question is to determine $F(G)$. 
    
    For any $\phi\in F(G)$ its differential $d\phi_e$ is a linear map on $\mathfrak g$. Let $F(\mathfrak g)$ the set of $d\phi_e$ for  $\phi \in F(G)$.

    A local isometry  is  a map $G\to G$ such that  at the identity element  $e$ is a local  diffeomorphism $\Psi':V_1 \to V_2$, where $V_1, V_2$ are neighborhhods of $e\in g$ and the  differential $d\Psi'$ preserves the metric at every point of $V_1$.  To compute local isometries, M\"uller proved the next result.

    \begin{thm}[\cite{MU} Theorem 2.2] 
    Let $(G, \lela\,,\,\rira)$ denote a Lie group with a bi-invariant metric.     Let $A$ be a linear endomorphism of $\mgg$. Then there exists a local isometry $\Phi$ of $G$ at $e$ such that $d \Phi_e = A$ if and only if $A$ satisfies the following conditions:
        \begin{enumerate}
            \item $\lela AX, AY \rira = \lela X, Y \rira$, for all   $X,Y \in \mgg$
            \item $A([X,[Y,Z]]) = [AX,[AY,AZ]]$ for all $X,Y,Z \in \mgg$.
        \end{enumerate}
        
    \end{thm}

The aim now is to give the isometry group of $\oscn$. As explained the essential point is to determine the isotropy subgroup $F(\oscn)$. Set 

- $\mathfrak f(\oscn)$ the Lie algebra of $F(G)$,

- $F(\mathfrak{osc}_n(\lambda_1, \hdots, \lambda_n))$ the group of isometries of the bilinear form on $\mathfrak{osc}_n(\lambda_1, \hdots, \lambda_n)$, that is 
$$Q_e = dtdz + \sum_{j=1}^n \frac{1}{\lambda_j}(dx_j^2 + dy_j^2).$$
- $\mathfrak f(\mathfrak{osc}_n(\lambda_1, \hdots, \lambda_n))$ the Lie algebra of $F(\mathfrak{osc}_n(\lambda_1, \hdots, \lambda_n))$. 

Since $\oscn$ is simply connected, the following map is a isomorphism (see \cite{MU})
$$\phi \in F(\oscn) \quad \to \quad d\phi_e\in F(\mathfrak g).$$
Moreover the group $F(\mathfrak g)$ consists of the linear maps $A:\mathfrak g\to \mathfrak g$ satisfying the conditions of the theorem above. 

Bourseau in \cite{Bou} studied to group $F(\mathfrak{osc}_n(\lambda_1, \hdots, \lambda_n))$. In the next paragraphs we reproduce the main  information of that work. Let $\rho$ denote the following matrix of $GL(2n, \mathbb R)$:
$$ \rho = \left( 
\begin{matrix} 
\lambda_1 & & & & \\
& \lambda_1 &   & &\\
& & \ddots & & \\
& & & \lambda_n & \\
& & & & \lambda_n
\end{matrix}\right).$$
Let $p\in \mathbb N$ with $0:=n_0 < n_1 < \hdots < n_p:=n_p$ such that
$$\rho_{\nu}:=\lambda_{n_{\nu - 1}+ 1}= \hdots = \lambda_{n_{\nu}}\quad \mbox{ for } \nu=1, \hdots, p$$ and let $m_{\nu}:=n_{\nu}- n_{{\nu}-1}$. 
    
    \begin{prop} For the bi-invariant metric in \eqref{metricosc}, let $A\in F(\mathfrak g)$. Then for $\nu=1, \hdots, p$, the map  $A$ has a matrix in the basis $Z, \{X_i, Y_i\}, T$ of $\mathfrak{osc}_n(\lambda_1, \hdots, \lambda_n)$:
    	$$\varepsilon  \left( 
    	\begin{matrix} 
    	 1 &	c_1^{\tau}	&\hdots  &c_p^{\tau}   &   -\frac12 \sum_{\nu=1}^p  \rho c_{\nu}^{\tau} c_{\nu}\\
    0	 &	B_1 & & &  - \rho_1 B_1 c_1\\
    \vdots	&	& \ddots &   &  \vdots\\
    0	&	& & B_p &  - \rho_p B_1 c_p\\
   0 & & & 0 & 1
    	\end{matrix}\right), \quad \mbox{where } \varepsilon=\pm 1, c_{\nu}\in \mathbb R^{2m_{\nu}}, B_{\nu}\in \mathrm{O}(2 m_{\nu}).$$
    \end{prop}

Below, one describes the structure of $F(\mathfrak{osc}_n(\lambda_1, \hdots, \lambda_n))$. Recall that $\mathrm{O}(1)=\{-1,1\}$. 

\begin{defn} Let $K$ be the compact group 
	$$K= \mathrm{O}(1)\times \times_{\nu=1}^p \mathrm{O}(2m_{\nu}).$$
	Define the semidirect product
	$$F=K\ltimes_{\pi} \mathbb R^{2n}$$
	where 
	$$\pi(\varepsilon, B_1, \hdots, B_p)(c)=\left( \begin{matrix}
	B_1  & & \\
	 & \ddots & \\
	  & & B_p
	\end{matrix} \right) c.
	$$	
	\end{defn}
\begin{prop} The map $\Psi: F(\mathfrak{osc}_n(\lambda_1, \hdots, \lambda_n)) \to F$ given by
	$$\Psi \left( \varepsilon  \left( 
	\begin{matrix} 
	1&	c_1^{\tau}	&\hdots  &c_p^{\tau}   &   -\frac12 \sum_{\nu=1}^p  \rho c_{\nu}^{\tau} c_{\nu}\\
0&	B_1 & & &  - \rho_1 B_1 c_1\\
\vdots &	& \ddots &   &  \\
0&	& & B_p &  - \rho_p B_1 c_p\\
0  & &\hdots  & 0 & 1
	\end{matrix}\right) \right) =(\varepsilon, B_1, \hdots, B_p, c)$$
	is a isomorphism of Lie groups. 
\end{prop}

To describe the structure of $F$, we identify the interior automorphisms, that we introduced as conjugation maps. Let $I_h$ denote the interior automorphism, with $d_e I_h=Ad(h): \mathfrak{osc}_n(\lambda_1, \hdots, \lambda_n)$. If we denote by $h=(z,v,t)$, clearly $I_h=I_{(0,v,t)}$.

Let $P_{\lambda_i}(t)\in \mathrm{SO}(2)$ define by
$$P_{\lambda_i}(t):=\left\{
	 \begin{array}{cl}
\left( \begin{matrix}
	\sin(t\lambda_i) & 1 -\cos(t\lambda_i)\\
-1 +\cos(t\lambda_i)  & \sin(t\lambda_i)
\end{matrix}\right) & \mbox{ for } t\in \mathbb R -\{\frac{2m\pi}{\lambda_i}, m\in \mathbb Z\}\\
 \left( \begin{matrix}
1 &0\\
0  & 1
\end{matrix}\right) & \mbox{ for } t\in\{\frac{2m\pi}{\lambda_i}, m\in \mathbb Z\}
\end{array}	
\right.
$$
and define $P(t)\in \mathrm{SO}(2n)$ by
$$P(t)=\left( \begin{matrix}
P_{\lambda_1}(t) & & \\
& \ddots & \\
& & P_{\lambda_n}(t)
\end{matrix}\right).$$

\begin{lem} Consider the following element on $F(\mathfrak{osc}_n(\lambda_1, \hdots, \lambda_n))$:
	$$\left( \begin{matrix}
	1 & & \\
	& B & \\
	& & 1
	\end{matrix}\right).$$
Then there exists a isometry $\Theta(B):\oscn \to \oscn$ given by
$\Theta(B)(z,v,t)=(z, P(t)^{\tau}BP(t)v, t)$ with $d \Theta(B)_e=\left( \begin{matrix}
1 & & \\
& B & \\
& & 1
\end{matrix}\right).$
\end{lem}

Let $K_1(\oscn)$ define by
$$K_1(\oscn)=\{\Theta(B):\oscn \to \oscn \,  /\,$$
$$\qquad \qquad \qquad  \qquad \qquad \qquad \Theta(B)(z,v,t)=(z,P(t)^{\tau}BP(t)v,t)\, B_{\nu}\in\mathrm{O}(2m_{\nu})\}$$
and let $$K(\oscn)=K_1(\oscn)\cup sK_1(\oscn),$$
where $s:\oscn \to \oscn, \, s(g)=g^{-1}$ the inversion map.
Moreover 
$$K_1(\oscn)_0=\{\Theta(B)\in K_1(\oscn)\,/\, B_{\nu}\in \mathrm{SO}(2m_{\nu})\}.$$
Recall that the conjugation map $I{(z,v,t)}$ depends only on $(v,t)$ and this set $\{(v,t): v \in \mathbb R^{2n}, t\in \mathbb R\}$ with the structure $\oscn/\{(z,0,0)\}_{z\in\mathbb R}$ is a solvable Lie group of dimension $2n+1$. Denote by $Int(\oscn)$ the set of inner automorphisms, which is a subgroup of the isometry group. 

Let $s:\oscn \to \oscn$ denote the inversion  isometry:
$$s: (z,v,t) \quad \to \quad (z,v,t)^{-1}=(-z, -R(-t)v,-t).$$

    \begin{thm}\cite{Bou} \label{Bou}
        The subgroup of isometries fixing the identity element has the next structure
        $$F(\oscn)=K(\oscn)\cdot 
       Int(\oscn),$$
       $ \, \mbox{ with  }\, K(\oscn)\cap Int(\oscn)=\{id\}.$
        
        Furthermore, $Int(\oscn)$ is a normal subgroup in $F(\oscn)$ and it holds
    $$ \begin{array}{crcl}
        	(i) & \Theta(B)\circ I_{(v,t)}\circ \Theta(B)^{-1} & = &I_{(JBJ^{\tau}v,t)};\\  	
        
        	(ii) & s\circ I_{(v,t)}\circ s^{-1} & = & I_{(v,t)};\\
        
        (iii) & s \circ \Theta(B) \circ s^{-1} &   = & \Theta(B).
        \end{array}$$
  $F(\oscn)$ consists of $2^{p+1}$ connected components and for the connected component of the identity one has
  $$F(\oscn)_0=K_1(\oscn)_0\cdot Int(\oscn).$$
        \end{thm}

As corollary one has that the Lie algebra of $F(\oscn)$ is a Lie algebra isomorphic to the following one
$$\mathfrak f ={\Large \times}_{\nu=1}^p \mathfrak{so}(2m_{\nu}) \ltimes \mathbb R^{2n}.$$
    Let $Aut(\oscn)$ denote the group of automorphisms of the oscillator group and let $Sp(m_{\nu})$ denote the group of $2m_{\nu}\times 2m_{\nu}$-symplectic matrices over $\R$. Now, we would like to determine, which isometries are automorphisms. Denote by $\tilde{K}_1$ the compact subgroup of $K_1(\oscn)$ given by 
     $$\tilde{K}_1=\{\Theta(B)\in K_1(\oscn)/B_{\nu}\in \mathrm{O}(2m_{\nu})\cap Sp(2m_{\nu})\}.$$
	Take $M\in \mathrm{O}(2n)$ given by
	$$M=\left( \begin{matrix}
	1 & 0 & & &\\
	0 & -1 & & &\\
	& & \ddots & &\\
	& & & 1 & 0\\
	& & & 0 & -1
	\end{matrix}
	\right).
	$$
	\begin{prop}
		Let $\tilde{K}$ be the compact subgroup of $K(\oscn)$ given by
	$\tilde{K}=\tilde{K}_1 \cup s\circ \Theta(M) \circ \tilde{K}_1$, then it holds
	$$F(\oscn)\cap Aut(\oscn)=\tilde{K}\cdot Int(\oscn).$$
	\end{prop}
Finally Bourseau studied the isometry group of $\oscn$. He found that
the Lie algebra of $Iso(\oscn)$ is the semidirect product
$$\mathfrak{iso}(\oscn)=(\times_{\nu=1}^p \mathfrak{so}(2m_{\nu}))\ltimes \mathfrak g_{2n},$$
where $\mathfrak{g}_{2n}$ is oscillator algebra of dimension $4n+2$. 

Thus, the isometry group follows 
$$Iso(\oscn)=L(\oscn) F(\oscn)$$
 and it holds $L(\oscn)\cap F(\oscn)=\{id\}.$

\begin{rem}
	Note that since the inversion map $h\to h^{-1}$ is a isometry of the Lie group  $(\oscn, \lela\,,\,\rira)$, then the compact spaces $\Lambda \backslash \oscn$ and $\oscn/\Lambda$ are isometric. In fact the map $x\Lambda \to \Lambda x^{-1}$ is a isometry between both spaces. Clearly $\oscn$ acts on $\oscn/\Lambda$ on the left transitively.
\end{rem}

\subsection{Isometries in the quotients. } The aim now is to study the isometry group of the quotient spaces. Let $\Lambda$ denote a discrete subgroup of $\oscn$ such that  $M=\oscn/\Lambda$ is a compact space. Since the metric on $\oscn$ is both, right and left-invariant, it can be induced to the cosets $g\Lambda\in M$. Indeed a isometry of $M$ gives rise to a local isometry in $\oscn$, since the projection $\oscn \to M$ is a submersion which is a local isometry. Thus, one has a local isometry of $\oscn$ satisfying conditions in the paragraphs above. 

On the other hand, some isometries of $\oscn$ can be induced to the quotient. The next result explicit conditions of such maps. 

\begin{defn} Let $f$ be  an isometry of $\oscn$, and $\Lambda$ a lattice. We say that $f$ is {\em fiber preserving} if $f(g)^{-1} f(g\lambda) \in\Lambda$ for all $g \in \oscn$, and for every $\lambda\in \Lambda$. 
\end{defn}
		If $f$ is a fiber preserving
isometry, it induces an isometry  $\tilde{f}$ on the compact space $M=\oscn\backslash \Lambda$   by defining $\tilde{f}(g\Lambda) = f(g)\Lambda$.

\begin{obs} Note the following facts. 
	\begin{itemize}
		\item Translations on the left by elements of group, are fiber preserving maps. Every map $L_h$ induces the isometry $\tau_h$ in $M$. In particular $L_{\lambda}(\Lambda)\subseteq \Lambda$ for every $\lambda\in \Lambda$. Denote by  $\tilde{L}(M)=\{\tau_g:g\in \oscn\}$. 
		\item If $f$ is a fiber preserving map that fixes the identity, then $f(\lambda)\in \Lambda$, for every $\lambda$ in the lattice $\Lambda$. 
	\end{itemize}
	
\end{obs}

Recall that whenever the lattices $\Lambda_1$ and $\Lambda_2$ are not pairwise isomorphic, they determine non-diffeomorphic solvmanifolds (see for instance \cite{Ra}).

One can study the isometry group of $\oscn/\Lambda$ once one has information about the isometry group of $\oscn$. The following result is consequence of the Lifting theorem. The proof can be seen in \cite{BOV}.

\begin{thm} Let $G$ be an arcwise-connected, simply connected Lie group with a bi-invariant metric and $\Lambda$ be a discrete subgroup
of $G$. Then every isometry $f$ of $G/\Lambda$ is induced to $G/\Lambda$ by a fiber preserving isometry of $G$.
\end{thm}

In view of that we proceed to study the fiber preserving isometries of $G$, specifically, those in the isotropy subgroup.

Analogously as in \cite{BOV}, computations show that neither the inversion map $s$, nor the map $\Theta(B)$ are fiber preserving.
In fact, to see that assume $\lambda=(\tilde{z},\tilde{v},\tilde{t})\in \Lambda$.  By computing 
		$(z,v,t)(\tilde{z},\tilde{v},\tilde{t})^{-1}(-z,-R(-t)v,-t)$ and looking at the component in $\mathbb R^{2n}$, one obtains 
				
		$v-R(-\tilde{t})v-R(t-t_1)\tilde{v},$ 
		
		which must belong to $\Lambda\cap \mathbb R^{2n}$ for every $v\in  \mathbb R^{2n}$ and $t\in \mathbb R$. In particular for $v=0$ one gets that for all $t\in \mathbb R$, it holds $R(t-t_1)\tilde{v}\in \Lambda\cap \mathbb R^{2n} \subset \Lambda$ which is countable set. This is contradiction. And analogously with $\Theta(B)$. 

Thus, it rest to determine which inner autormophisms are fiber preserving. Let $I_h: \oscn \to \oscn$ the conjugation map, then
$$I_h(g)^{-1}I_h(g\lambda)=hg^{-1}h^{-1} hg\lambda h^{-1}= h\lambda h^{-1}\in\Lambda,$$ for every $\lambda\in \Lambda$. The condition above says that  
  $h \in N_G(\Lambda)$, the normalizer of the lattice $\Lambda$ in $\oscn$. 

Since any isometry $f$ of the Lie group can be written as $f=L_p\circ g$ with $g$ is a isometry fixing the isometry element, we have that any isometry in the quotient space $M=\oscn/\Lambda$ can be written as $\tilde{f}=\tau_g\circ \tilde{h}$, where $\tilde{h}$ denotes the isometry induced to the quotient by $h$, $\tilde{h}(g\Lambda)=h(g)\Lambda$.

Consider the following homomorphisms where $G=\oscn$:
\begin{itemize}
	\item $\widetilde{I}:N_G(\lambda) \to Iso(M)$ given by $\widetilde{I}(h)=\widetilde{I}_h$ and
	\item $\tau: G \to Iso(M)$ which gives $\tau(g)=\tau_g$.
\end{itemize}

By Isomorphism Theorem, one has $\tilde{L}(M)=Im\tau=G/\ker\tau$ and $\ker\tau$ contains the elements in the intersection of the center and $\Lambda$: $Z(\oscn)\cap \Lambda$, implying that $\tau$ is not injective. 

On the other hand it is easy to see for any $h\in Z(\oscn)$ one has $I_h(x)=x$, so that $h\in \ker \widetilde{I}$. And in this case $\widetilde{I}$ is not injective. 

In any case, to specify those statements one needs more information about $\Lambda$.

\begin{prop}
 Let $\Lambda$ be a lattice in $\oscn$. The isometries in the Lie group that are fiber preserving correspond to translations on the left by elements of the group and the inner automorphisms $I_h$ for $h\in N_G(\Lambda)$. Moreover, 
 any isometry $f$ in $M=\oscn/\Lambda$ can be written as $f=\tau_g\circ \widetilde{I}_h$ but this  is not necessarily unique. 
\end{prop}

\begin{rem} For the subgroups $\Lambda$ in Example \ref{Lattice4}, it was proved in \cite{BOV}  that 
	
	$\tilde{L}(M)\cap Im \tilde{I}=\{\tau_Z\circ\widetilde{I}_{\lambda}, \mbox{ for } Z\in Z(G), \lambda\in \Lambda_{k,s}\}$. 
	\end{rem}
	
	\begin{exa}
	In dimension four one may consider the lattices $\Lambda_{k,0}, \Lambda_{k,\pi}, \Lambda_{k,\pi/2}$ in the oscillator group $G$ of dimension four with $\lambda_1=1$. See Example \ref{Lattice4}. As said in \cite{BOV}, since the subgroups $\Lambda_{k,j}$ are not pairwise isomorphic, the corresponding compact spaces are not homeomorphic.

To	compute the normalizers of such lattices, one may  may find the elements $(z,v,t)$ such  that
	$(z,v,t)(\tilde{z},\tilde{v},\tilde{t})(-z,-R(-t)v,-t)\in \Lambda$, for  all $(\tilde{z},\tilde{v},\tilde{t})\in \Lambda$, where $\Lambda$ is a lattice. The proof follows by writing down the coordinates. 
	\begin{itemize}
			\item For $\Lambda_{k,0}$ the map $R(nt_0)$ is the identity for $t_0=2\pi$. Thus, one has $v+R(t)\tilde{v}-R(\tilde{t})v\in \mathbb Z^2$, implies that $R(t)=s\frac{\pi}2$ for $s\in \mathbb Z$.
			
			From the $z$-coordinate one has 
			
			\smallskip
			
			$\tilde{z}+\frac12 v^{\tau} JR(t)\tilde{v} -\frac12 v^{\tau} J R(\tilde{t})v-(R(t)\tilde{v})^{\tau}JR(\tilde{t})v= \frac{u}{2k}$ for some $u\in \mathbb Z$,
			
			\smallskip
			
			which implies $v^{\tau}JR(t)\tilde{v}=\frac{\tilde{u}}{2k}$ for some $\tilde{u}\in \mathbb Z$. Therefore $v\in \frac{1}{2k}\mathbb Z^2$. 
			
			So, we get $N_G(\Lambda_{k,0})=\mathbb R\times \frac{1}{2k}\mathbb Z^2\times \frac{\pi}2 \mathbb Z$. 
			
			\item For $\Lambda_{k,\pi}$ the map $R(2n\pi)$ is the identity  or $R(n\pi)=-Id$ for $n$ odd.
			
			A similar reasoning as above gives $N_G(\Lambda_{k,\pi})=\mathbb R\times \frac{1}{2}(\mathbb Z)^2\times \frac{\pi}2 \mathbb Z$.	
		 \item For $\Lambda_{k,\pi}$ the map $R(n\pi/2)=\pm Id$ if $n$ is even with $-Id$ if $n\equiv 2 mod(4)$. Thus a reasonning as above says that
			$N_G(\Lambda_{k,\pi/2})=\mathbb R\times (\mathbb Z)^2 \times \frac{\pi}2 \mathbb Z$.
		\end{itemize}
	In dimension six the situation is much more complicated, as we show below. 
	\end{exa}
	\subsection{An example in dimension six.} Assume we have the Lie group $Osc(1,\lambda)$ which has the differentiable structure of $\mathbb R^6$. As said in Lemma \ref{lema_medina}, to have a cocompact lattice, one needs that the real numbers $1,\lambda$ generate a discrete subgroup of $\mathbb R$. 
		On the other hand it is known that a subgroup $H$ of $\mathbb R$ is either discrete or dense. Moreover, if it is discrete, then $H=\mathbb Z r$ being $r=inf(H\cap \mathbb R_{>0})>0$. Thus, we are  in  the last situation. 
	This implies that there exist $n\in \mathbb Z$ such that $1=nr$ saying that $r\in \mathbb Q$. Analogously, since there exists $s\in \mathbb Z$ such that $sr=\lambda$ we have that $\lambda \in \mathbb Q$. 
	
Thus for the corresponding Lie group we have for some $r\in \mathbb Q$, the map $R(t)$ has a matrix presentation as follows:
$$R(t)=\left( \begin{matrix} 
\cos(t) & -\sin(t) & 0 & 0\\
\sin(t) & \cos(t) &0 & 0\\
0 & 0 & \cos(r t) & -\sin(rt)\\
0 & 0 & \sin(rt) & \cos(rt)
\end{matrix}
\right)$$
 
 Note that if $r=p/q$ with $p$ and $q$ coprime numbers, then $t_0=q$ generates a subgroup of $\mathbb Z$ and both $\cos(2sq\pi)(m),\sin(2sq\pi)(m)\in \mathbb Z$ for every $s,m\in \mathbb Z$ and also $\cos(2sp\pi)(m),\sin(2sp\pi)(m)\in \mathbb Z$ for all $s,m\in \mathbb Z$. 
 
 Thus, $R(2st_0\pi)(\mathbb Z\times \mathbb Z\times \mathbb Z\times \mathbb Z)\subseteq \mathbb Z\times \mathbb Z\times \mathbb Z\times \mathbb Z$, which says that 
    
 $\Lambda_{k,0}=\frac1{2k} \mathbb Z\times \mathbb Z^4 \times 2q\pi\mathbb Z$ 
 
 is a cocompact lattice of $\mathrm{Osc}(1,p/q)$ for $k\in \mathbb N$.
 
 Analogously, one proves that the following set is also a cocompact lattice:
 
  $\Lambda_{k,\pi}=\frac1{2k} \mathbb Z\times \mathbb Z^4 \times q\pi\mathbb Z$. 
  
  In this case $\cos(sq\pi)(m)=\pm m,\sin(sq\pi)(m)=\pm m$ depending on the parity of $sq$, and  $\cos(2sp\pi)(m)=\pm m,\sin(2sp\pi)(m)=\pm m$  depending on the parity of $sp$. In any case, $\cos(2sq\pi)(m),\sin(2sq\pi)(m)\in \mathbb Z$. A similar reasoning applies for  the lattice 
  
  $\Lambda_{k,\pi/2}=\frac1{2k} \mathbb Z\times \mathbb Z^4 \times q\frac{\pi}2 \mathbb Z$.
  
  In this way, we generalize the lattices considered in \cite{BOV}, for every $k\in \mathbb N$, to 
  $$\Lambda_{k,q,M} = \frac{1}{2 k} \Z \times \Z^4 \times \frac{2 \pi q}{M} \Z, \quad \mbox{for} \, M=1,2,4.$$

 Note that one obtains three distinct families for $k, q \in \NN$  when:
 \begin{enumerate}
     \item $\forall q, k$, when $M=1$,
     \item $q$ odd, when $M=2$,
     \item $q$ odd,  when $M=4$
 \end{enumerate}

Given $g=(z,v,t) \in Osc_2(1,p/q)$ and $\lambda=(a,\eta,b) \in \Gamma$ any of the lattices above, computing $g \lambda g^{-1}$ gives
\begin{equation}
    (a + \frac{1}{2}v^T J R(t) \eta - \frac{1}{2} v^T J R(b) v - \frac{1}{2} \eta^T R(t)^t J R(b) v, v + R(t) \eta - R(b) v, b)
\end{equation}.

In order to have $g \Lambda g^{-1} \subset \Lambda$, one gets the conditions on $z,v,t$ as follows:

\begin{align}
    v + R(t) \eta - R(\frac{2 \pi q}{M} c)v &\in \Z^4   \label{c1} \\ 
    \frac{1}{2} \left[ v^T J R(t) \eta - v^T J R(b) v - \eta^T R(t)^TJR(\frac{2 \pi q}{M} c) v \right] &\in \ZZ/{2k}, \label{c2}
\end{align}
 for any $ \eta \in \RR^4$, and $ b \in \frac{2 \pi q}{M} \Z$ with $c = 0, ..., M-1$. \\

For the case $c=0$ in Equation (\ref{c1}) leads to $R(t) \eta \in \Z^4$, and consequently 
\begin{equation}
t \in \frac{q \pi}{2} \ZZ. \label{cc1}
\end{equation}

Similarly $c=0$ in Equation (\ref{c2}) implies that $$ v^T J R(t) \eta \in Z / {2 k} \,\, \forall \eta \in \ZZ^4 $$     
and consequently 
\begin{equation}
    v \in \ZZ^4/{2k}. \label{cc2}
\end{equation}

For $c = 1, ..., M-1$, Equations (\ref{c1}) and (\ref{c2}) reduce to

\begin{align}
    v - R(\frac{2 \pi q}{M}c) v &\in \Z^4  \label{ccc1} \\
    (v^T + v^T R^T(\frac{2 \pi q}{M}c)) J R(t) \eta - v^T J R(\frac{2 \pi q}{M}c) v  &\in \ZZ/{k} \nonumber
\end{align}

The latter equation can be simplified noticing that $\eta = 0$ implies 
\begin{equation}
    v^T J R(\frac{2 \pi q}{M}c)v \in \Z/k, \label{ccc2}
\end{equation}
and therefore $(v^T + v^T R^T(\frac{2 \pi q}{M}c)) J R(t) \eta \in \Z / k$, 
which is equivalent to:

\begin{equation}
    v + R(\frac{2 \pi q}{M}c) v \in \Z^4 / k \label{ccc3}
\end{equation}

Consider the following set definitions, which are useful for the next lemma and proposition:
$$\mathcal{C} := \{ (p,q,k,M): p,q,k \in \NN, \mbox{ p and q coprime}, M \in \{ 1,2,4\}, \mbox{ q odd when $M > 1$} \},$$
$$F := \frac{1}{2}\{ z \in \ZZ : z \,\, odd \},$$
$$I_2 := \ZZ^2 \cup F^2,$$
$$I_4 := \ZZ^4 \cup F^4.$$

\begin{lem}
     Let $(p,q,k,M) \in \mathcal{C}$, then $(z, (v_1,v_2,v_3,v_4),t)$ is an element of the normalizer of $\Lambda_{k,q,M}$ in $Osc(1,p/q)$ if it satisfies the next conditions

    \def\arraystretch{1.5}
    \begin{center}
    \begin{tabular}{|c|c|}\hline
    \multirow{2}{*}{M=1,2,4} & $t \in \frac{q \pi}{2} \mathbb{Z}$   \\\cline{2-2}
        & $(v_1,v_2,v_3,v_4) \in \ZZ^4 / {2 k}$  \\\hline
    \multirow{2}{*}{M=2} & $(v_1,v_2) \in \ZZ^2/2$  \\\cline{2-2}
        & $(v_3,v_4) \in \ZZ^2/2$ if p odd   \\\hline
    \multirow{4}{*}{M=4} & $(v_1,v_2) \in I_2$  \\\cline{2-2}
        & $(v_1,v_2) \in \ZZ^2 \, \text{if p even and k odd}$  \\\cline{2-2}
        & $(v_3,v_4) \in I_2 \, \text{if p odd}$  \\\cline{2-2}
        & $(v_1,v_2,v_3,v_4) \in I_4 \, \text{if p odd and k odd}$  \\\hline
    \end{tabular}
    \end{center}
    
\end{lem}

\begin{proof}
    As seen above, given $(z,v,t)$ in the normalizer, then $v$ must satisfy Equations (\ref{cc1}) and (\ref{cc2}) for any $M=1,2,4$ and Equations (\ref{ccc1}), (\ref{ccc2}) and (\ref{ccc3}) for $M=2,4$. For this proof $A=(v_1,v_2)$ and $B=(v_3,v_4)$ are defined.\\

    For $M=2$ and $M=4$, notice that the subspaces $\R^2 \times {(0,0)}$ and ${(0,0)} \times \R^2$ are invariant under the linear operators $J$ and $R(\frac{2 \pi q}{M}c)$ and therefore the linear equations (\ref{ccc1}) and (\ref{ccc3}) result equivalent to 

    \begin{align}
    A& - r(\frac{2 \pi q}{M}c) A \in \ZZ^2 \label{ab1}\\
    B& - r(\frac{2 \pi p}{M}c) B \in \ZZ^2 \label{ab2}\\
    A& + r(\frac{2 \pi q}{M}c) A \in \ZZ^2/k \label{ab3}\\
    B& + r(\frac{2 \pi p}{M}c) B \in \ZZ^2/k \label{ab4},
    \end{align}

where $r(t) := \left( \begin{matrix} 
\cos(t) & -\sin(t)\\
\sin(t) & \cos(t)
\end{matrix}
\right)$. With this in mind and recalling from Section \ref{section1} that $J_1 = \left( \begin{matrix} 
0 & -1\\
1 & 0
\end{matrix}
\right)$, Equation (\ref{ccc2}) can be written as follows

\begin{equation}
    A^{\tau} J_1 r(\frac{2 \pi q}{M}c) A + B^{\tau} J_1 r(\frac{2 \pi p}{M}c) B \in \ZZ/k. \label{ab5}
\end{equation}

\underline{For $(M=2, c=1)$ and $(M=4, c=2)$}: given that $r(\frac{2 \pi n}{M}c) = \pm Id $, Equations (\ref{ab1}) to (\ref{ab4}) result in:

\begin{align}
    2A& \in \ZZ^2 \text{ if q odd} \label{ab6}\\
    2B& \in \ZZ^2 \text{ if p odd} \label{ab7}\\
    2A& \in \ZZ^2/k \text{ if q even} \nonumber \\
    2B& \in \ZZ^2/k \text{ if p even}. \nonumber
\end{align}

Notice that the last two equations are already satisfied by condition: $v \in \ZZ^4/{2k}$. Notice also that Condition (\ref{ab5}) is trivial since the expression is null.\\

\underline{For $M=4, c=1,3$}: Equations (\ref{ab1}) to (\ref{ab4}) can be expressed as follows

\begin{align}
    A& \pm J_1 A \in \ZZ^2 \text{ if q odd} \label{ab8}\\
    B& \pm J_1 B \in \ZZ^2 \text{ if p odd} \label{ab9}\\
    A& \pm J_1 A \in \ZZ^2/k \text{ if q odd} \nonumber \\
    B& \pm J_1 B \in \ZZ^2/k \text{ if p odd}, \nonumber
\end{align}

the last two equations are weaker than the first two so only (\ref{ab8}) and (\ref{ab9}) need to be considered. When condition (\ref{ab8}) applies, it also applies condition (\ref{ab6}) and the solution is $A \in I_2$. Analogously with (\ref{ab7}) and (\ref{ab9}), the solution is $B \in I_2$.

It remains to add the conditions that result from (\ref{ab5}). From now one, assume $q$ odd. Observe that only for $M=4, c=1,3$ the equation is not trivial. Consider first $p$ even, the condition can be expressed  as
$$ v_1^2 + v_2^2 \in \ZZ/k. $$

Since $(v_1,v_2) \in  I_2$ the condition above is valid always when $k$ is even, while if $k$ is odd, $(v_1,v_2) \in \ZZ^2$ is required. 

When $p$ is odd, the condition can be expressed as one of the following two forms, (depending on the precise values of $p$ and $q$).

$$ v_1^2 + v_2^2 \pm (v_3^2 + v_4^2) \in \ZZ/k.$$

Similar to the previous case, if $k$ is even the  condition holds since $(v_1,v_2), (v_3,v_4) \in I_2$. If $k$ is odd $v \in I_4$ is required.

\end{proof}

By using this information we are able to compute the normalizers. 

\begin{prop} The normalizers of the lattices $\Lambda_{k,q,M} \in Osc(1,p/q)$, for $(p,q,k,M) \in \mathcal{C}$, are given in the next table

\begin{center}
\def\arraystretch{1.5}
\begin{tabular}{ |c|c|c| } 
\hline
M & Conditions & Normalizer = $N(\Lambda_{k,1,q})$ \\
\hline
\multirow{1}{1em}{1} & - & $\R \times \Z^4 / {2k} \times q \frac{\pi}{2} \Z$ \\ 
\hline
\multirow{2}{1em}{2} & p even & $\R \times \Z^2/2 \times \Z^2 / {2k} \times q \frac{\pi}{2} \Z$ \\ 
&  p odd & $\R \times \Z^4/2 \times q \frac{\pi}{2} \Z$ \\ 
\hline
\multirow{4}{1em}{4} & p even, k even & $\R \times I_2 \times \Z^2 / {2k} \times q \frac{\pi}{2} \Z$ \\ 
&  p even, k odd & $\R \times \Z^2 \times \Z^2 / {2k} \times q \frac{\pi}{2} \Z$ \\ 
&  p odd, k even & $\R \times I_2 \times I_2 \times q \frac{\pi}{2} \Z$\\ 
&  p odd, k odd & $\RR \times I_4  \times q \frac{\pi}{2} \ZZ$ \\ 

\hline
\end{tabular}
\end{center}

\end{prop}

This result is straightforward from the previous Lemma when combining all conditions.

	\appendix 


\begin{thebibliography}{GGGG}
		
		\bibitem{BG} {\sc O. Baues, W. Globke}, {\it Rigidity of compact pseudo-Riemannian homogeneous spaces for solvable Lie groups}. 
			Int. Math. Res. Not. {\bf  2018} (10), 3199--3223 (2018). 
		\bibitem{Be} {\sc A. Beardon}, {\it The geometry of discrete groups}. Springer (1983). First Edition. 
		
		\bibitem{Bou} {\sc F. Bourseau}, {\it Die Isometrien der Oszillatorgruppe und einige ergebnisse \"uber Pr\"amorphismen liescher Algebren}. Fakult\"at f\"ur Mathematik der Universit\"at Bielefeld (1989).
		
			\bibitem{BOV} {\sc V. del Barco, \sc G. Ovando, \sc F. Vittone}, {\it Lorentzian compact manifolds: Isometries and geodesics}, J. Geom. Phys. {\bf 78}, 48--58 (2014).
		
		
			\bibitem{MF} {\sc M. Fischer}, {\it Lattices of oscillator groups}, J. Lie Theory {\bf 27} (1), 85--110 (2017). 	
			
			\bibitem{Ga} {\sc G.  Galloway}, {\it  Compact Lorentzian manifolds without closed nonspacelike geodesics}, Proc.
			Amer. Math. Soc. {\bf 98}, 119--123  (1986).
			
			\bibitem{Hel} {\sc S. Helgasson}, {\it Differential Geometry, Lie Groups, and Symmetric Spaces}, Graduate Studies in Mathematics, vol. {\bf 34}, American Math. Soc. (1999).
			
		\bibitem{Me} {\sc A. Medina}, {\it Groupes de Lie munis de m\'etriques bi-invariantes}. (Lie groups admitting bi-invariant metrics), T\^ohoku Math. J., II. Ser. {\bf 37}, 405--421 (1985). 
		
		
	\bibitem{MeRe} {\sc A. Medina,  P. Revoy}, {\it Les groupes oscillateurs et leurs r\'eseaux. (Oscillator groups and their lattices).}, Manuscr. Math. {\bf 52}, 81--95 (1985). 
	
	
	
	
	\bibitem{MU} {\sc D. M\"uller}, {\it Isometries of bi-invariant pseudo-Riemannian metrics on Lie groups},  Geom. Dedicata {\bf 29} (1),  65--96 (1989).
	
		
		
		\bibitem{ON} {\sc B. O'Neill}, {\it Semi-Riemannian geometry with
			applications to relativity}, Academic Press (1983).
		
		

		
		\bibitem{Ov} {\sc G. Ovando}, {\it Lie algebras with ad-invariant metrics- A survey}, In Memorian Sergio Console, Rendiconti del Seminario Matematico di Torino. {\bf  74}, 1-2, 241 -- 266 (2016).
		
		\bibitem{Ra} {\sc  M.S. Raghunathan}, Discrete Subgroups of Lie Groups, Springer Verlag, 1972.
		
		\bibitem{Su} {\sc S. Suhr}, {\it Closed geodesics in Lorentzian surfaces},  	Trans. Amer. Math. Soc. {\bf 365}, 1469--1486 (2013).
		
		
		
	\end{thebibliography}
\end{document}